\documentclass{amsart}
\usepackage{amssymb}
\usepackage{graphicx}
\usepackage{epstopdf}
\usepackage{hyperref}
\usepackage{enumitem}
\theoremstyle{plain}
\newtheorem{theorem}{Theorem}[section]
\newtheorem{corollary}[theorem]{Corollary}

\newtheorem{lemma}[theorem]{Lemma}
\newtheorem{proposition}[theorem]{Proposition}
\theoremstyle{definition}
\newtheorem{definition}[theorem]{Definition}
\theoremstyle{example}
\newtheorem{example}{Example}
\theoremstyle{remark}
\newtheorem{remark}{Remark}

\begin{document}
    \title[Unbounded disjointness preserving functionals and operators]{Unbounded Disjointness Preserving Linear Functionals and  Operators}
\author{Anton~R. Schep}
\address{Department of Mathematics\\
         University of South Carolina\\
	 Columbia, SC 29212\\
	 USA}
\email{schep@math.sc.edu}
\urladdr{http://www.math.sc.edu/~schep/}

\begin{abstract}
Let $E$ and $F$ be Banach lattices.  We show first that the disjointness preserving linear functionals separate the points of any infinite dimensional Banach  lattice $E$, which shows that in this case the unbounded disjointness operators from $E\to F$ separate the points of $E$.
Then we show that every disjointness preserving operator $T:E\to F$ is norm bounded on an order dense ideal.  In case $E$ has order continuous norm, this implies that that every unbounded disjointness preserving map $T:E\to F$ has a unique decomposition $T=R+S$, where $R$ is a bounded disjointness preserving operator and $S$ is  an unbounded disjointness preserving operator, which is zero on a norm dense ideal. For the case that $E=C(X)$, with $X$ a compact Hausdorff space, we show that every disjointness preserving operator $T:C(X)\to F$ is norm bounded on an norm dense sublattice algebra of $C(X)$, which  leads then to a decomposition of $T$ into a bounded disjointness operator and a finite sum of unbounded disjointness preserving operators, which are zero on  order dense ideals. \end{abstract}

\maketitle
\section{Introduction.} The topic of automatic continuity of linear operators having additional structure has been studied studied in many contexts. Most famous is probably the cumulative result of Dales, Esterle, and Woodin (see \cite{Dales1987}), who showed that there are models of ZFC in which every algebra homomorphism between certain Banach algebras is automatically continuous and assuming ZFC with AC (i.e., the continuum hypothesis) that there exist discontinuous algebra isomorphisms. One way one can then proceed is to weaken the condition of being an algebra homomorphism. This has be done for what sometimes are called separating maps, i.e., a linear mapping $T$ between algebras is called separating if $T(a)T(b)=0$, whenever $ab=0$ (see e.g. \cite{Jarosz1990} ). If we restrict ourselves to algebras of continuous functions, then this condition is equivalent to the condition that $T$ is a disjointness preserving map.  In this context set theoretic models don't play a role anymore. The first explicit construction of an unbounded disjointness preserving functional  seems to be in the paper by Abramovich (\cite{Abramovich1983}). He showed that  unbounded disjointness preserving functionals exist on $C(K)$, where $K$ is an infinite compact Hausdorff space.  It follows from a result of de Pagter ( \cite{Pagter1984a}) that in fact there exist unbounded disjoint linear functionals on any infinite dimensional Banach lattice. All these constructions essentially consist of constructing prime ideals, which are not maximal. Brown and Wong \cite{Brown2004} made a detailed study of unbounded disjointness linear functionals on spaces $C(K)$, but were apparently not aware of de Pagter's paper ( \cite{Pagter1984a}). As all examples of unbounded disjointness preserving operators use unbounded disjointness preserving functionals, we will make first a systematic study of unbounded disjointness preserving functionals, generalizing the results of Brown and Wong \cite{Brown2004}. Our approach was inspired by the paper of Abramovich and Lipecki \cite{Abramovich1990},but our emphasis is different. They were concerned with the cardinality of the set of non-order bounded disjointness preserving linear functionals, while we study here the separating properties of the disjointness preserving linear functionals.  Next we study norm unbounded disjointness preserving operators.  We show that every disjointness preserving operator $T$ from a Banach lattice $E$ into a Banach lattice $F$  is norm bounded on an order dense ideal. This implies that in that case there exists a norm closed, order dense ideal $I_{T}$ in $E$ such that the restriction $T_{|_{I_{T}}}$ of $T$ can be written as a sum $R+S$, where $R$ is a bounded disjointness preserving operator from $I_{T}$ to $F$ and $S$ is  an unbounded disjointness preserving operator from $I_{T}$ to $F$, which is zero on a order dense ideal in $I_{T}$. In case $E$ has order continuous norm, we have then such a decomposition on $E$, as obviously $I_{T}=E$.  Then, following the ideas of the Bade and Curtis theorem for algebra homomorphism as given in \cite{Dales1987}, we find in the case that $E$ has a strong unit more detailed information on the ``purely unbounded'' part $S$ of $T$.
\section{Unbounded disjointness preserving linear functionals.}
In this section we use the observation of de Pagter \cite{Pagter1984a} (Proof of Theorem 10) that whenever $M$ is proper prime ideal and $x\notin M$, then there exist a disjointness preserving $\phi$ on $E$ such that  $\phi(x)=1$ and $M\subset \ker(\phi)$. In case the prime ideal has infinite co-dimension in $E$, then the disjointness preserving functional $\phi$ can be chosen to be not order bounded. To construct prime ideals with some desired properties we will use repeatedly the Prime ideal Separation Theorem for Riesz spaces, due to Wim Luxemburg. Recall first from \cite{Zaanen83} the notion of a filter in a Riesz space. A subset $F$ of a Riesz space $E$ is called a \emph {filter} if
\begin{enumerate}
\item $F\neq \emptyset$,
\item $0\notin F$, and $f\in F$ if and only if $|f|\in F$,
\item If $0\le u, v\in F$, then $u\wedge v\in F$.
\item $0<u\le v$, $u\in F$ implies that $v\in F$.
\end{enumerate}
We now state the Prime ideal Separation Theorem for Riesz spaces. In \cite[Theorem 76.3]{Zaanen83}  the corresponding theorem for distributive lattices is stated explicitly and proved in detail. Then the actual theorem for Riesz spaces is only mentioned  in a remark. We will include for completeness the proof, which is essentially the same as \cite[Theorem 76.2(i)]{Zaanen83}.

\begin{theorem}[Prime Ideal Extension Theorem]
Let $E$ be a Riesz space,  $M$ an ideal in $E$ and $F$ a filter such that $M\cap F=\emptyset$. Then there exists a prime ideal $M_{0}\supset M$ such that $M_{0}\cap F$.
\end{theorem}
\begin{proof}
By Zorn's Lemma there exist an ideal maximal with respect to the properties that $M_{0}\supset M$ and $M_{0}\cap F=\emptyset$. We need to show that $M_{0}$ is prime. To this end let $y,z\in E$ such that $y\wedge z=0$ and assume $y,z\notin M_{0}$. Then $(M_{0}\vee A_{y})\cap F\neq \emptyset$, so there exist $0\le m_{1}\in M_{0}$ and $0\le y_{1}\in A_{y}$ such that $m_{1}\vee y_{1}\in F$. Similarly there exist $0\le m_{2}\in M_{0}$ and $0\le z_{1}\in A_{z}$ such that $m_{2}\vee z_{1}\in F$. Let $m=m_{1}\vee m_{2}$. Then $m\in M_{0}$ and $m\vee y_{1}, m\vee z_{1}\in F$. Now $y_{1}\wedge z_{1}=0$ implies that $m=m\vee(y_{1}\wedge z_{1})=(m\vee y_{1})\wedge (m\vee z_{1})$ and thus $m \in F$. This is a contradiction and thus either $y\in M_{0}$ or $z\in M_{0}$. It follows that $M_{0}$ is prime.
\end{proof}
\begin{theorem}\label{main}
Let $E$ be a non-zero Archimedean Riesz space. Then for all $0\neq u\in E$ there exists a disjointness preserving linear functional on $E$ such that $\phi(u)=1$.
\end{theorem}
\begin{proof}
Let $0\neq u\in E$. We need to consider two cases. Assume first that the dimension of the principal ideal $A_{u}=\{x\in E: |x|\le \lambda u\}$ generated by $u$ is of finite dimension. Then there exist atoms $e_{1}, \cdots, e_{n}$ such that  $u=\sum_{k=1}^{n}e_{k}$. Moreover $E=A_{u}\oplus \{u\}^{d}$. If $v\in A_{u}$, then $v$ can be uniquely written as $\sum_{k=1}\alpha_{k}e_{k}$. Define $\phi(v)=\alpha_{1}$ and extend $\phi$ to $E$ by defining $\phi=0$ on $\{u\}^{d}$. Clearly $\phi$ is now a disjointness preserving functional such that $\phi(u)=1$. Next we consider the case that $A_{u}$ is infinite dimensional. In that case there exist non-zero disjoint $u_{n}$ such that $0<u_{n}\le u$. Define \[F=\{x\in E: |x|\ge u_{n} \text{ for all } n\ge N_{x}\}.\]
Then $F$ is a filter in $E$ and $F\cap M=\emptyset$, where $M=\{x\in E: |x|\le \lambda (u_{1}+\cdots+u_{n}) \}$ is the ideal generated by the $u_{n}$'s. By the Prime ideal Extension theorem there exists a prime ideal $M_{0}\supset M$ such that $M_{0}\cap F=\emptyset$. As $u\notin M_{0}$, there exists a disjointness preserving functional $\phi$ with  $\phi(u)=1$.
\end{proof}
\begin{remark}
The idea of the above proof  comes from the proof of Theorem 4 of \cite{Abramovich1990}. There the ideal and filter were lifted into the lattice of ideals of $E$ and then the prime ideal separation theorem was applied in that setting. As the above proofs show, it is actually simpler to do this directly in the Riesz space setting.

In case 1 of the above proof the disjointness preserving linear functional is automatically order bounded (and norm bounded if $E$ is a normed Riesz space), but even in case 2 this can happen as the following example illustrates.
\end{remark}
\begin{example}\label{example}
Let $E=c_{00}+\mathbb R(1,1,\cdots)=\{x=(x_{n}): x_{n}=c \text{ for all }n\ge N_{x}\}$. Then $\phi_{n}(x)=x_{n}$ and $\phi_{\infty}(x)=\lim x_{n}$ are disjointness preserving linear functionals on $E$ and every disjointness preserving linear functional on $E$ is of the form $c\phi_{n}$ or $c\phi_{\infty}$, for some real constant $c$. Note that in this case the principal ideal generated by $(1,1,\cdots)$ is infinite dimensional.
\end{example}
In this paper we are concerned about the existence of non-order bounded disjointness preserving linear functionals,  and thus norm unbounded in case of a normed Riesz space. To this end we observe that a disjointness preserving functional $\phi$ on $E$ is order bounded if and only if for all $0\le v\le u$ we have that $|\phi(v)| \le |\phi(u)|$. 
\begin{theorem}
Let $E$ be a relative uniformly complete Riesz space. Then for all  $x\in E$ such that the principal ideal $A_{|x|}$ is infinite dimensional, there exists a non-order bounded disjointness preserving $\phi$ such that $\phi(x)=1$.
\end{theorem}\begin{proof} Let $x_{0}>0$ such that $A_{x_{0}}$ is infinite dimensional. Then there exist disjoint non-zero $0\le x_{n}\le x_{0}$ for all $n\ge 1$. Let $y_{n}=\frac {x_{n}}{2^{n}}$. Then the series $\sum y_{n}$ converges uniformly to some $y$ with $0\le y\le x_{0}$.  Let $M$ be the ideal generated by the $y_{n}$'s. The filter $F$ is now defined as above, i.e.,  $F=\{x\in E:   |x|\ge y_{n} \text{ for all } n\ge N_{x}\}$. In particular $y\in F$ and $cx_{0}\in F$ for all $c>0$. Then again we have $M\cap F=\emptyset$, so by the the Prime ideal Separation Theorem there exists a prime ideal $M_{0}\supset M$ such that $M_{0}\cap F=\emptyset$. denote by $[x]$ the equivalence class of $x$ in $E/M_{0}$. We claim now $[x_{0}]$ and $[x_{0}+y]$ are linearly independent. If not, then there exist $c\in \mathbb R$ and $z\in M_{0}$ such that $cx_{0}+(x_{0}+y)=z$, so $(c-1)x_{0}+y=z$. This implies $c\neq 1$, since otherwise $y=z\in M_{0}\cap F=\emptyset$.   Now $c\neq 0$ implies that there exists $N$ such that
\[\sum_{n=N}^{\infty} y_{n}\le \frac {|c-1|}2 x_{0}.\]
Let $z_{0}= z-\sum_{n=1}^{N}y_{n}$. Then $z_{0}\in M_{0}$ and
\[ |z_{0}|\ge |c-1|x_{0}-\sum_{n=N}^{\infty}y_{n}\ge \frac {|c-1|}2 x_{0}.\]
This shows that $z_{0}\in F$, which contradicts $M_{0}\cap F=\emptyset$
Hence  there exists a disjointness preserving linear functional $\phi$ with $\phi(x_{0})=1$ and $\phi(x_{0}+y)=0$. Since $0\le x_{0}\le x_{0}+y_{0}$, this implies that $\phi$ is non-order bounded.
\end{proof}

The following corollary was proved for $c_{0}$ by Brown and Wong \cite{Brown2004} using completely different methods.
\begin{corollary}
For each sequence $x$ in $c_{0}$ (or $\ell_{p}$ with $1\le p\le \infty$) which is not eventually null, there is an unbounded disjointness preserving linear functional $\phi$ on $c_{0}$ (or $\ell_{p}$) such that $\phi(x)=1$.
\end{corollary}
 It was shown by de Pagter \cite{Pagter1984a} that every non-order bounded disjointness preserving linear map is order bounded on an order dense ideal. We refine this result for non-order bounded disjointness linear functionals.
\begin{theorem}
Let $E$ be an Archimedean Riesz space and assume $\phi$ is a non-order bounded disjointness preserving linear functional on $E$. Then $\phi$ is zero on an order dense prime ideal. \end{theorem}
\begin{proof}
Let $\phi$ be a non-order bounded disjointness preserving linear functional on $E$.  Denote by $\mathcal U=\{u: u\ge 0, A_{u} \text{ is infinite dimensional }, \phi(x)=0 \text{ for all } x\in A_{u}\}$. To show that $\mathcal U\neq \emptyset$, observe first that if $u_{n}$ ($n=1, 2, \cdots$) are non-zero disjoint elements in $E$, then $\phi(u_{n})\neq 0$ for at most one $n$. Let now $0<u\in E$ such that $A_{u}$ is infinite dimensional. Note such $u$ exist, otherwise $\phi $ would be order bounded. Then there exist disjoint $0<v_{n}\le u$ ($n=1, 2, \cdots$). We claim there exists an $n$ such that $\phi$ is zero on $A_{v_{n}}$. If not, there exist $x_{n}\in A_{v_{n}}$ such that $\phi(x_{n})\neq 0$ for all $n\ge 1$. This contradicts the observation that all but one of the $\phi(x_{n})=0$. Let now $I=\{x: |x|\le C(u_{1}+\cdots +u_{n}), u_{i}\in \mathcal U\}$ denote the ideal generated by all the $u\in \mathcal U$. Then $\phi$ is zero on $I$. Next we need  to show that $I$ is order dense. If $I^{d}$ is infinite dimensional, then as above we can find a non-zero principal ideal in $I^{d}$ such that $\phi=0$ on that principal ideal. This contradicts the definition of $I$. Therefore $I^{d}$ is finite dimensional. If the dimension of $I^{d}$ is greater or equal to 1, then the restriction of $\phi$ to $I^{d}$ is a non-zero order bounded disjointness preserving linear functional. Then $\phi$ is equal to zero on $I^{dd}$ and thus $\phi$ is order bounded on $E$. This contradiction shows that $I^{d}=\{0\}$ and thus $I$ is order dense. Remains to show that $I$ is prime. Let $x,y\in E$ and assume $x\wedge y=0$. Then either $\phi(x)=0$ or $\phi(y)=0$. Assume $\phi(x)=0$ and $x\notin I$. Then there exists $0<z\in A_{x}$ such that $\phi(z)\neq 0$. if now $0<w\in A_{y}$, then $z\wedge w=0$. This implies $\phi(w)=0$, as $\phi(z)\neq 0$, Hence $A_{y}\subset I$ and thus $y\in I$. This shows that $I$ is prime. \end{proof}

We conclude this section with an extension theorem for disjointness preserving functionals.
\begin{theorem}
Let $E$ be an Archimedean Riesz space. Let $0<x\in E$ and assume $\psi: A_{x}\to \mathbb R$ is an order bounded disjointness preserving linear functional. Then there exists a disjointness preserving linear functional $\phi: E\to \mathbb R$ such that $\phi_{|_{A_{x}}}=\psi$.
\end{theorem}
\begin{proof} Without loss of generality we can assume that $\psi(x)=1$. Define then $M=\{y\in A_{x}: \psi(y)=0\}$ and $F=\{y\in E: |y|\ge x\}$. Then $M$ is an ideal in $E$, $F$ is a filter in $F$ and $M\cap F=\emptyset$. By the prime ideal extension theorem there exists a prime ideal $M_{0}\supset M$ such that $M_{0}\cap F=\emptyset$.  Hence there exists a disjointness preserving linear functional $\phi:E\to \mathbb R$ such that $\phi$ is zero on $M_{0}$ and $\phi(x)=1$. Now $\phi_{|_{A_{x}}}$ is zero on $M$ implies that there exists a constant $c$ such that $\phi_{|_{A_{x}}}=c\psi$. Evaluating both sides of this equation at $x$ we see that $c=1$ and the proof is complete.

\end{proof}

\begin{remark}
We note that the extension $\phi$ is no longer order bounded in many cases. E.g. If $E=L_{2}[0,1]$ w.r.t. Lebesgue measure and $x$ is the function identical one on $[0,1]$, then there exists a positive disjointness preserving linear functional $\psi$ on $A_{x}=L_{\infty}[0, 1]$ corresponding to the delta functional  at $0$. As there exist no order bounded disjointness linear functionals on $E$, the extension $\phi$ of $\psi,$ given by the above theorem, can't be order bounded.
\end{remark}
\section{Unbounded disjointness preserving operators on  Banach lattices}
Let $E$ and $F$ be Banach lattice and $T:E\to F$ be a disjointness preserving operator. Then $T$ is norm (and thus order) bounded if and only if $\|T(v)\|\le \|T(u)\|$ for all $|v|\le |u|$. For unbounded disjointness preserving $T$ we still have some inequality in the following situation.
\begin{lemma} Let $E$ and $F$ be Banach lattice and $T:E\to F$ be a disjointness preserving operator. 
If $u=v+w$ and $v\perp w$, then $\|T(v)\|\le \|T(u)\|$.
\end{lemma}  
\begin{proof} Let $u=v+w$ and $v\perp w$. Then $Tv\perp Tw$, so $|Tu|=|Tv|+|Tw|$ implies that $\|Tv\|\le \|Tu\|$.

\end{proof}
The following theorem is the analog of the Main Boundedness Theorem of Bade and Curtis (\cite{Bade1960}) in their study of unbounded algebra homomorphisms (see also \cite[Theorem 1.3]{Dales1987}).
\begin{theorem}\label{main}
Let $E$ and $F$ be Banach lattices and assume $T: E\to F$ is a linear disjointness preserving operator. Then for all disjoint sequences $u_{n}$ in $E$, there exists a constant $C$ such that $\|T(u_{n})\|\le C\|u_{n}\|$.
\end{theorem}
\begin{proof}
Let $u_{n}$ be disjoint elements in $E$ with $\|u_{n}\|\le 1$. Assume there is no $C$ such that $\|T(u_{n})\|\le C$ for all $n$. Then by passing to a subsequence we can assume that $\|T(u_{n})\|\ge 4^{n}$. Let $u=\sum_{n=1}^{\infty}\frac {u_{n}}{2^{n}}$. Then $\frac {u_{n}}{2^{n}}\perp \sum_{k\neq n}\frac {u_{k}}{2^{k}}$ implies that $\|T(u)\|\ge \|T(\frac {u_{n}}{2^{n}})\|\ge 2^{n}$ for all $n$, which is a contradiction.
\end{proof}
The proof of the following theorem uses the above theorem repeatedly.
\begin{theorem} Let $E$ and $F$ be Banach lattices and assume $T: E\to F$ is a linear disjointness preserving operator. Then $T$ is norm bounded on an order dense ideal $E_{T}$.

\end{theorem}
\begin{proof}
Let $0\le u_{n}$ be a disjoint collection of elements. Then it follows from the above theorem that the restriction $T_{|_{A_{u_{n}}}}$ of $T$ to the principal ideal $A_{_{u_{n}}}$ is bounded for all but finitely many $n$. Moreover by the same theorem $\|T_{|_{A_{u_{n}}}}\|$ is uniformly bounded for all the $n$ such that $T_{|_{A_{u_{n}}}}$ is bounded. Let now $\mathcal U$ be a maximal disjoint set of $0< u\in E$ such that $T_{|_{A_{u}}}$ is bounded. Let $E_{T}$ denote the ideal generated by $u\in \mathcal U$. Again by the above theorem and the maximality of $\mathcal U$ it follows that $E_{T}^{d}=\{0\}$. Remains to show that $T$ is norm bounded on $E_{T}$. Assume to the contrary that there exist $0\le x_{n}\in E_{T}$ with $\|x_{n}\|\le1$ such that $\|Tx_{n}\|\to \infty$ as $n\to \infty$. We will construct  a disjoint  sequence $\{w_{n}\}$ in $E_{T}$ with the same properties. We take $w_{1}=x_{1}$. Now $0\le w_{1}\le c_{1}u_{1}$, where $u_{1}$ is a finite supremum of elements of $\mathcal U$. For $n\ge 2$ we can now find $u_{n}$, again finite suprema of elements of $\mathcal U$, disjoint with $u_{1}$, and $c_{n}, d_{n}$ such that $0\le x_{n}\le c_{n}u_{1}+d_{n}u_{n}$. Hence for all $n\ge 2$ we can write $x_{n}=y_{n}+z_{n}$, where $0\le y_{n}\le c_{n}u_{1}$ and $0\le z_{n}\le d_{n}u_{n}$. Now $\|Ty_{n}\|\le \|T_{|_{A_{u_{1}}}}\|$ for all $n\ge 2$ implies that $\|Tz_{n}\|\to \infty$. Let $w_{2}=z_{k}$, where $\|Tz_{k}\|\ge 2$. Note $w_{1}\perp w_{2}$. Now repeat this construction with the sequence $\{w_{1}+ w_{2}, z_{3}, \cdots\}$. As $\|Tz_{n}||\to \infty$ we can split as above the $z_{n}$ for $n\ge 3$. Continuing this way we end up with a disjoint sequence $w_{n}$ with $\|w_{n}\|\le 1$ and $\|Tw_{k}\|\ge k$ for all $k\ge 2$, which contradicts the above theorem. Hence $T$ is bounded on $E_{T}$.
\end{proof}
\begin{remark}
It was proved in \cite{Pagter1984a}, in a more general context, that every disjointness preserving linear map is order bounded on an order dense ideal. The above corollary does however not follow from this, as the dense order ideal is in general not norm complete. In fact it is a consequence of the above corollary that every disjointness preserving operator $T$ is order bounded on the order dense ideal $E_{T}$.
\end{remark}
\begin{theorem}\label{decomposition}
Let $E$ and $F$ be Banach lattices and let $T: E\to F$ be a linear disjointness preserving operator. Then then there exists an order dense closed ideal $I_{T}$ in $E$ such that the restriction $T_{|_{I_{T}}}$ of $T$ to $I_{T}$ can be written as $T=R+S$, where $R$ is a bounded disjointness preserving operator on $I_{T}$ with $T=R$ on $E_{T}$ and $S$ is a disjointness preserving operator on $I_{T}$ such that $S=0$ on $E_{T}$. Moreover the $R$ and $S$ are uniquely determined by these properties.
\end{theorem}
\begin{proof} By the above proof $T$ is norm bounded on the order dense ideal $E_{T}$. Let $I_{T}$ be the norm closure of $E_{T}$ in $E$. Then $T_{|_{E_{T}}}$ has a unique norm bounded extension $R$ defined on  $I_{T}$. Let $S=T-R$ on $I_{T}$. Then we need to show that $S$ is disjointness preserving. Assume first that $f\in E_{T}$ and $g\in I_{T}$ such that $f\perp g$. Then $Rf\perp Rg$ and $Tf\perp Tg$. As $Rf=Tf$ this implies that $Tf\perp Tg-Rg=Sg$. Assume now $f,g\in I_{T}$ such that $f\perp g$. Then we can find $f_{n}\in E_{T}$ with $|f_{n}|\le |f|$ such that $f_{n}\to f$ in norm. Now $f_{n}\perp g$ implies that $Rf_{n}=Tf_{n}\perp Sg$ for all $n$ and by continuity it follows that $Rf\perp Sg$. By symmetry we also get $Sf\perp Rg$. Now $|Sf|\wedge |Sg|= |Tf-Rf|\wedge |Sg|=|Tf|\wedge |Sg|=|Tf|\wedge |Sg-Tg|=|Tf|\wedge |Rg|=|Tf-Rf|\wedge |Rg|=|Sf|\wedge |Rg|=0$. Hence $S$ is disjointness preserving. To prove uniqueness assume that also $T=R_{1}+S_{1}$, where $R_{1}$ is a bounded disjointness preserving operator on $I_{T}$ with $T=R_{1}$ on $E_{T}$ and $S_{1}$ is a disjointness preserving operator on $I_{T}$ such that $S_{1}=0$ on $E_{T}$. Then $S-S_{1}=R_{1}-R$ is bounded  operator on $I_{T}$, which is zero on the norm dense ideal $E_{T}$. Hence $S=S_{1}$ on $I_{T}$, which proves the uniqueness.

\end{proof}
In case $E$ has order continuous norm, then $E_{T}$ is norm dense in $E$, so that $I_{T}=E$. Hence we have the following corollary. 
\begin{corollary}Let $E$ and $F$ be Banach lattices and assume $E$ has order continuous norm. Let $T: E\to F$ be a linear disjointness preserving operator. Then $T=R+S$, where $R$ is a bounded disjointness preserving operator on $E$ with $T=R$ on $E_{T}$ and $S$ is a disjointness preserving operator on $E$ such that $S=0$ on $E_{T}$. Moreover the $R$ and $S$ are uniquely determined by these properties.

\end{corollary}
\begin{remark}
The above theorem reduces the study of non-bounded disjointness preserving operators to those who are zero on a norm dense ideal. A first result is the following proposition.
\end{remark}
\begin{proposition}Let $E$ and $F$ be Banach lattices and assume $E$ has order continuous norm. Let $T: E\to F$ be a linear disjointness preserving operator such that $T=0$ on an order dense ideal. Then for all disjoint sequences $\{u_{n}\}$ in $E^{+}$ we have that $Tu_{n}=0$ except at most finitely many $n$.

\end{proposition}
\begin{proof}
Let $0<u_{n}$ be disjoint elements in $E$. Then by the argument of the above corollary we have that the restriction $T_{|_{A_{u_{n}}}}$ of $T$ to the principal ideal $A_{_{u_{n}}}$ is bounded for all but finitely many $n$. As $T$ is zero on an order dense ideal in $E$, this implies that  $T_{|_{A_{u_{n}}}}$   is zero for all but finitely many $n$.
\end{proof}
\begin{remark}
An inspection of the last proof shows that with a modification we have a much stronger conclusion. If $T$ and $u_{n}$ are as in the proposition, then $T_{|_{\{u_{n}\}^{dd}}}=0$ except at most finitely many $n$.
\end{remark}

\section{Unbounded disjointness preserving operators  on $C(X)$ spaces}
In this section $X$ will denote a compact Hausdorff space. For $x\in X$ we denote by $\mathcal N_{x}$ the collection of open neighborhoods of $x$. We first state a topological lemma. For the elementary proof, see  the discussion at the bottom of page 10 of \cite{Dales1987}
\begin{lemma}\label {regular}
Let $X$ be a regular topological space and $\{x_{n}\}$ a sequence of distinct elements in $X$. Then there exist disjoint  $U_{n}\in \mathcal N_{x_{n}}$.
\end{lemma}
As every compact Hausdorff is completely regular, the above lemma holds for compact Hausdorff spaces. For an open subset $U$ of $X$ we associate the closed lattice and algebra ideal $K_{U}$ of $C(X)$ determined by $U^{c}$, i.e., $K_{U}=\{f\in C(X): f(x)=0 \text{ for all }x\in U^{c}\}$. Similar  to \cite{Dales1987} we now introduce the singular points of a disjointness preserving operator $T:C(X)\to F$, where $F$ is a Banach lattice.
\begin{definition}A point $x\in X$ is called a \emph{singularity} point of a disjointness preserving operator $T:C(X)\to F$ if for all $U\in \mathcal N_{x}$ the operator $T_{\vert_{K_{U}}}:K_{U}\to F$ is discontinuous. 

\end{definition}
We denote by $\mathcal F_{T}$ the set of all singularity points of $T$ and we will call this set the singularity set of $T$. We introduce two additional notations. For $x\in X$ we denote by $M_{x}$ the maximal ideal determined by $\{x\}$, i.e., $M_{x}=\{f\in C(X): f(x)=0\}$ and by $J_{x}$ the lattice and algebraic ideal of all $f\in C(X)$, which vanish on some open neighborhood of $x$. Note that the uniform closure of $J_{x}$ equals $M_{x}$ and that $J_{x}$ is order dense in $C(X)$ for any non-isolated $x\in X$. Moreover $J_{x}$ a lattice and algebraic prime  ideal in $C(X)$ for any non-isolated $x\in X$. If $Y\subset X$ we will denote by $J_{Y}$ the set of all $f\in C(X)$ which vanish on a neighborhood of $Y$.
The following theorem is the main result of this section. It generalizes the structure theorem of algebra homomorphisms of Bade-Curtis (\cite{Bade1960}).
\begin{theorem}\label{Bade}
Let $X$ be a compact Hausdorff space, $F$ be a Banach lattice and $T:C(X)\to F$ a disjointness preserving operator. Then the following hold.
\begin{enumerate}
\item The singularity set $\mathcal F_{T}=\emptyset$ if and only if $T$ is continuous.
\item The singularity set $\mathcal F_{T}$ is a finite set.
\end{enumerate}
Assume now that the  singularity set $\mathcal F_{T}\neq \emptyset$ and that $\mathcal F_{T}=\{x_{1}, \dots, x_{n}\}$. 
\begin{enumerate}[resume]
\item There exist a continuous disjointness preserving operator $R:C(X)\to F$ and disjointness preserving operators $S_{i}:C(X)\to F$ ($1\le i\le n$) such that $T=R+S_{1}+\dots S_{n}$. Moreover ${S_{i}}_{|_{M_{x_{i}}}}\neq 0$ and ${S_{i}}_{|_{J_{x_{i}}}}= 0$ for each $1\le i\le n$.
\end{enumerate}

\end{theorem}
\begin{proof}
For the proof of (1) it is clear that if $T$ is continuous, then $\mathcal F_{T}=\emptyset$. Conversely, if $\mathcal F_{T}=\emptyset$, then for all $x\in X$ there exists $U_{x}\in \mathcal N_{x}$ such that $T_{\vert_{K_{U_{x}}}}:K_{U_{x}}\to F$ is continuous. By compactness of $X$ there exists a finite  open cover $\{\ U_{x_{i}}: 1\le i\le n\}$ of $X$ such that $T_{\vert_{K_{U_{x_{i}}}}}:K_{U_{x_{i}}}\to F$ is continuous for each $i$. Using now a partition of unity subordinate to this open cover we conclude that $T$ is continuous. To prove (2) assume that $\mathcal F_{T}$ is infinite. Then using Lemma \ref{regular} we can find a sequence $x_{n}\in \mathcal F_{T}$ and $U_{n}\in \mathcal N_{x_{n}}$ such that $U_{n}\cap U_{m}=\emptyset$ for all $n\neq m$. Since $T_{\vert_{K_{U_{n}}}}:K_{U_{n}}\to F$ is discontinuous, we can find $f_{n}\in K_{U_{n}}$ such that $\|Tf_{n}\|\ge n\|f_{n}\|$ for all $n$. As $f_{n}\perp f_{m}$ for all $n\neq m$ this contradicts the main boundedness theorem Theorem \ref{main}. For the proof of (3) we first prove that for all $x\in \mathcal F_{T}$ there exists $U_{x}\in \mathcal N_{x}$ such that $T_{|_{K_{U_{x}}}\cap J_{x}}$ is continuous. If this is not the case, we can find by induction $V_{n}\in  \mathcal N_{x}$, $f_{n}\in J_{x}$ such that $f_{n}\in K_{V_{n}}$, $f_{i}(V_{n+1})=0$ for $1\le i\le n$, and $\|Tf_{n}\|\ge n\|f_{n}\|$. This contradicts the main boundedness theorem Theorem \ref{main}. Applying this observation to $\mathcal F_{T}=\{x_{1}, \dots, x_{n}\}$ we can find open $U_{k}$ such that $T_{|_{K_{U_{k}}}\cap J(\mathcal F_{T})}$ is continuous for $1\le k\le n$. For each $x\in X\setminus \mathcal F_{T}$ we can find an open set $V$ such that $T_{|_{K_{V}}}$ is continuous. Using the compactness of $X$ we can therefore find a open cover consisting of the $U_{k}$'s and finitely many $V_{i}$ such that $T_{|_{K_{U_{k}}}\cap J(\mathcal F_{T})}$ is continuous for each $k$ and $T_{|_{K_{V_{i}}}\cap J(\mathcal F_{T})}$ is continuous for each $i$. Using then partition of unity argument we see that $T_{|_{J(\mathcal F)}}$ is continuous. Let $\mathcal B$ denote the set $\{f\in C(X): f-f(x_{i})\mathbf{1}\in J_{x_{i}}\}$, i.e., $f\in \mathcal B$ if $f$ is constant in a neighborhood of each $x_{i}\in \mathcal F_{T}$. It follows that $\mathcal B$ is lattice sub-algebra of $C(X)$, which is dense in $C(X)$ by the Stone-Weierstrass theorem. As in the proof of the Bade-Curtis structure theorem as given in \cite{Dales1987} we can show that $T$ is bounded on $\mathcal B$. Let $R:C(X)\to F$ denote the continuous extension of $T_{|_{\mathcal B}}$. Then $R$ is bounded disjointness preserving operator on $C(X)$, which extends $T_{|_{J({\mathcal F_{T}})}}$. Therefore if we define $S=T-R$, then $S_{|_{J(\mathcal F_{T})}}=0$. It follows now as in the proof of Theorem \ref{decomposition} that $S_{|_{\overline{J(\mathcal F_{T})}}}$ is disjointness preserving. To define the $S_{i}$ find disjoint open neighborhoods $W_{i}\in \mathcal N_{x_{i}}$, $e_{i}\in K_{W_{i}}$ such that $\|e_{i}\|=1$ and $e_{i}(x)=1$ on a neighborhood of $x_{i}$. Define then $S_{i}(f)=S(fe_{i})$ for $1\le i\le n$. From $\mathbf{1}-\sum_{i=1}^{n} e_{i}\in J(\mathcal F_{T})$ it follows that $f-\sum_{i=1}^{n}fe_{i}\in J(\mathcal F_{T})$ for all $f\in C(X)$. Hence $S(f-\sum_{i=1}^{n}fe_{i})=0$, i.e., $Sf=\sum_{i=1}^{n}S_{i}(f)$ for all $f\in C(X)$. To prove that $S_{i}$ is disjointness preserving, assume $f\perp G$ in $C(X)$. Then at least one of $f(x_{i})$ and $g(x_{i})$ must be equal to zero. If both are equal to zero, then $f,g\in M_{x_{i}}=\overline{J_{x_{i}}}$. This implies that $fe_{i}, ge_{i}\in \overline{J(\mathcal F_{T})}$. Then $fe_{i}\perp ge_{i}$ implies that $S(fe_{i})\perp S(ge_{i})$, i.e., $S_{i}(f)\perp S_{i}(g)$. Assume now that $f(x_{i})=0$, but $g(x_{i})\neq 0$. Then $g(x)\neq 0$ on a neighborhood of $x_{i}$, which implies that $f(x)=0$ on a neighborhood of $x_{i}$, i.e., $f\in J_{x_{i}}$. This implies that $fe_{i}\in J(\mathcal F_{T})$. Hence $S_{i}(f)=S(fe_{i})=0$ in this case, in particular $S_{i}(f)\perp S_{i}(g)$. This proves that each $S_{i}$ is disjointness preserving. Finally to show that ${S_{i}}_{|_{M_{x_{i}}}}\neq 0$, we assume the contrary. Then $S_{|_{K_{W_{i}}}}={S_{i}}_{{|_{K_{W_{i}}}}}=0$ and thus bounded, which implies that $T_{|_{K_{W_{i}}}}$ is bounded, which contradicts that $x_{i}$ is singularity point of $T$.
\end{proof}
\begin{remark}
The above proof  follows closely the proof of the Bade-Curtis theorem for algebra homomorphism as given in \cite{Dales1987}. There are however some subtle differences. In the algebra case the operators corresponding to our $S_{i}$ are only shown to be algebra homomorphisms when restricted to $M_{x_{i}}$ and no homomorphism properties are know for the corresponding operator $S$. Moreover in most cases where the theorems overlap the set $\mathcal F_{T}=\emptyset$. In fact if follows from Woodin's results (see \cite{Dales1987}) that there are models of set theory (without the continuum hypothesis) for which $\mathcal F_{T}=\emptyset$ for all algebra homomorphism defined on a $C(X)$ space. On the other hand on any infinite dimensional $C(X)$ space there are many discontinuous disjointness preserving operators, as indicated in the first part of this paper. The above theorem should also be compared with the result of Jarosz (\cite{Jarosz1990}), who studied disjointness preserving linear operators between $C(K)$ type spaces, but a major difference is that Jarosz's decomposition is done on the range side, while in our approach everything takes place at the domain side of the operator.
\end{remark}

It is natural question to compare the decomposition of this section with the one from the previous section. Let $T:E\to F$ be a disjointness preserving operator and assume that $E$ is a Banach lattice with order continuous norm. Then by the corollary to Theorem \ref{decomposition} we can decompose $T$ as $T=R+S$, where $R$ is a bounded disjointness preserving linear operator, which agrees with $T$ on an order dense ideal (so that $S$ is zero on an order dense ideal). Let $0<u\in E$. Then by Kakutani's theorem $A_{u}$ is lattice isomorphic to $C(K)$ with $K$ compact and Hausdorff. Then the restriction $T_{|_{A_{u}}}$ of $T$ to $A_{u}$ can then by Theorem \ref{Bade} be written as $R_{1}+S_{1}$, where $R_{1}$ is continuous with respect to $u$-relative uniform convergence. In particular $R_{1}$ is an order bounded disjointness preserving operator which agrees with $T$ on an order dense ideal of $A_{u}$. From this it follows by order continuity of the norm that $R$ agrees with $R_{1}$ on $A_{u}$. In case the norm on $E$ is not order continuous we could hope to prove the same for $u\in I_{T}$, where $I_{T}$ denotes the norm closed order dense ideal in $E$ on which we can decompose $T$ as $R+S$. The above argument breaks down however as the norm closure of the order dense ideal given by Theorem \ref{Bade} will remain a proper closed ideal in $A_{u}$ and we can't prove that the extension of $T$ to the algebra $\mathcal B$ introduced in the proof of Theorem \ref{Bade} is continuous with respect to the norm of $E$.  Summarizing the above discussion we have the following result.
\begin{theorem} Let $E$ and $F$ be Banach lattices and assume $E$ has order continuous norm. Let $T: E\to F$ be a linear disjointness preserving operator. Let $T=R+S$, where $R$ and $S$ are as in Theorem \ref{decomposition}. Let $0<u\in E$.  The restriction $T_{|_{A_{u}}}$ has by Theorem \ref{Bade} a decomposition $R_{1}+S_{1}$. Then $R_{1}=R_{|_{A_{u}}}$ and thus also $S_{1}=S_{|_{A_{u}}}$.

\end{theorem}

\end{document}